\documentclass[letterpaper, 10 pt, conference]{cls/ieeeconf}  
\IEEEoverridecommandlockouts                              
\makeatletter
\newcommand{\linebreakand}{%
  \end{@IEEEauthorhalign}\hfill\mbox{}\par\mbox{}\hfill\begin{@IEEEauthorhalign}}
\makeatother

\usepackage[table,dvipsnames]{xcolor}      
\usepackage[noadjust]{cite}

\usepackage{amsmath,amssymb,amsfonts,amsthm,dsfont,mathtools,graphicx,tabularx,adjustbox}
\graphicspath{{./fig/}}
\usepackage[font=small]{caption}
\usepackage[font=footnotesize]{subcaption}
\usepackage{balance}
\usepackage[breaklinks=true, colorlinks, bookmarks=true]{hyperref} 

\newtheorem{proposition}{Proposition}
\newtheorem{lemma}{Lemma}
\newtheorem{theorem}{Theorem}
\theoremstyle{definition}

\newtheorem*{problem*}{Problem}

\newtheorem*{remark*}{Remark}
\newtheorem{assum}{Assumption} 

\newcommand{\bbR}{\mathbb{R}}

\DeclareMathOperator*{\diag}{diag}

\title{\LARGE \bf
Safe Trajectory Tracking of the Stefan Problem with Second-Order Moving Boundary Dynamics
}

\author{Shumon Koga, Miroslav Krstic
\thanks{S. Koga is with the Department
of Computer Science and Systems Engineering at Kobe University,
Hyogo, Japan (e-mail: koga@harbor.kobe-u.ac.jp).} 
\thanks{M. Krstic is with the Department
of Mechanical and Aerospace Engineering, University of California at San Diego, La Jolla,
CA, 92093-0411 USA (e-mail: krstic@ucsd.edu).}
\thanks{This work was supported by JSPS KAKENHI Grant Number JP25058336.}}

\begin{document}

\maketitle
\thispagestyle{empty}
\pagestyle{empty}

\begin{abstract}
This paper considers a safe trajectory tracking of the Stefan problem with a second-order moving boundary dynamics. The model is given by a parabolic Partial Differential Equation (PDE) defined on a time-varying domain of moving boundary governed by a second-order Ordinary Differential Equation (ODE) associated with the Neumann boundary condition. A feedforward control is designed by a series expansion approach to solve the inverse Stefan problem under given reference trajectory of the moving boundary, and the convergence of infinite series is proven. A trajectory tracking controller is derived based on an energy-shaping, which ensures the safety of the model constraint in the closed-loop system. The closed-loop system is also shown to be globally exponentially stable with respect to the tracking error by performing PDE backstepping transformation and Lyapunov analysis. Numerical simulation illustrates an effective tracking performance of the proposed method under a sinusoidal reference trajectory. Code is released at \url{https://github.com/shumon0423/StefanTracking_ACC2026.git}.  
\end{abstract}


\section{INTRODUCTION} \label{sec:introduction}

The Stefan problem is a classical mathematical model representing physical phenomena of thermal phase change, such as melting or solidification, which is widely studied in many applications, including sea ice melting \cite{si2022coupled}, additive manufacturing \cite{wang2021closed}, continuous casting \cite{belhamadia2023numerical}, electrosurgery \cite{ran2024heat}, and cell therapy \cite{srisuma2025simulation}. The Stefan problem comprises a partial differential equation (PDE) for the temperature profile, defined over a time-varying spatial domain whose moving boundary is governed by an ordinary differential equation (ODE). This modeling approach, where the spatial domain grows or shrinks, is also known as a free boundary problem and has been incorporated even beyond thermal dynamics for applications such as modeling chemical reactions in batteries \cite{pozzato2024accelerating}, biological cell growth \cite{francis2025spatial}, and option pricing \cite{nwankwo2024deep}.

Although the Stefan problem has been studied for more than a century, the geometric nonlinearity arising from the coupling between the PDE and the moving boundary presents an inevitable challenge. Consequently, ongoing work has focused on studying the solution of the Stefan problem using state-of-the-art methods such as Bayesian optimization \cite{winter2023multi} and physics-informed machine learning \cite{wang2021deep}. Control design for the Stefan problem has been developed over the last two decades using several approaches, such as optimal control \cite{Hinze07}, model predictive control \cite{ecklebe2021model}, geometric control \cite{maidi2014}, and enthalpy-based control \cite{petrus12}. The PDE backstepping method, a systematic approach for designing boundary controllers and observers for many classes of PDEs and delay systems
as summarized in a recent survey paper \cite{vazquez2026backstepping}, 
has been successfully applied to the Stefan problem, ensuring global exponential stability \cite{Shumon19journal} and validated in physical experiments \cite{koga2020energy}. Several extensions of the backstepping design for the Stefan problem are provided in \cite{KKbook2021}. Other classes of PDEs with moving boundaries have also been handled by the backstepping method, including traffic shockwaves \cite{yu2020bilateral}, piston dynamics \cite{buisson2018control}, and neuron growth \cite{demir2024neuron}.

Trajectory tracking for the Stefan problem, which aims to follow a desired behavior of the moving boundary, has also been studied in a few works. A pioneering work on motion planning by series expansion \cite{dunbar2003motion} provides a solution for the reference temperature of a nonlinear Stefan problem for an a priori given moving boundary as a function of time. Trajectory tracking control designs to stabilize the Stefan system at a reference solution have been proposed using enthalpy-based feedback \cite{petrus2022solid} and a flatness-based approach \cite{ecklebe2021toward}. Moreover, \cite{galvao2022extremum} proposed an extremum-seeking control for the Stefan problem to drive the system to an unknown optimum by tracking a reference trajectory that includes sinusoidal perturbations on a cost map. However, in all those works, stability analysis has been performed under assumptions on either closed-loop performance or a static reference. The reasons for these assumptions are twofold: first, the requirement that a physical condition being satisfied under closed-loop control, and second, the spatial domain discrepancy between the system and reference profiles. As shown in \cite{ecklebe2021toward}, once the reference error PDE state is defined to maintain consistent conditions at the moving boundary, the error dynamics includes a time derivative of the reference state at the boundary. This term cannot be bounded in standard $L_2$ or $H_1$ norms for the classical Stefan problem with first-order moving boundary dynamics, as the velocity gap is equivalent in norm to the Neumann boundary value.

Ensuring that constraints are satisfied in control systems has recently been studied under the notion of safety and Control Barrier Function (CBF) \cite{ames2016control}. 
For trajectory tracking, combining a CBF-based safety filter with a preset planning and low-level control architecture was addressed in \cite{agrawal2024gatekeeper}. Based on an energy-shaping approach and passivity \cite{ortega2002putting}, an energy-based safety constraint and its associated CBF to preserve passivity were proposed in \cite{singletary2021safety} for handling kinematic constraints in robotic systems. Several extensions combining passivity with CBF-based safety have since been developed \cite{capelli2022passivity,califano2023passivity}.

Safe or constrained control of PDE systems has also been studied in several works. Early work by \cite{dubljevic2006predictive} proposed a model predictive control approach that incorporates state and input constraints through a modal decomposition of a parabolic PDE. Input constraints under saturation for unstable parabolic PDEs have been handled using linear matrix inequalities \cite{mironchenko2020local,yang2018output}. A Lyapunov-based constraint guarantee for PDEs has been addressed for fluid dynamics \cite{karafyllis2022spill} and batteries \cite{roy2024input}. The safe control design by a high-order CBF for PDEs was initially proposed for the Stefan problem with actuator dynamics in \cite{koga2023safe} as well as an event-triggered mechanism \cite{koga2023event}, and later for hyperbolic PDEs in \cite{wang2024safe}. A learning-based approach for the safe control of PDEs was proposed in \cite{hu2025safe} using a neural operator \cite{bhan2023neural}. 

This paper is the first to provide a rigorous proof of safety and tracking stability of the Stefan problem with a trajectory tracking control. We consider the Stefan problem with second-order moving boundary dynamics, proposed in \cite{koga2025cdc}, to represent thermal inertia in the phase change dynamics. The feedforward control is designed following the series expansion method, and convergence of the infinite series under the second-order moving boundary dynamics is also shown. The trajectory tracking control is then designed via energy-shaping with respect to the energy function of the reference solution, and the safety condition is proven to hold by maximum principle. Based on the reference PDE state proposed in \cite{ecklebe2021toward}, the stability analysis in the tracking error is employed using a Lyapunov approach and a PDE backstepping transformation. Numerical simulation is performed with a sinusoidal reference trajectory.
\section{PROBLEM STATEMENT} \label{sec:problem}

\begin{figure}[t]
\centering
\includegraphics[width=0.7\linewidth]{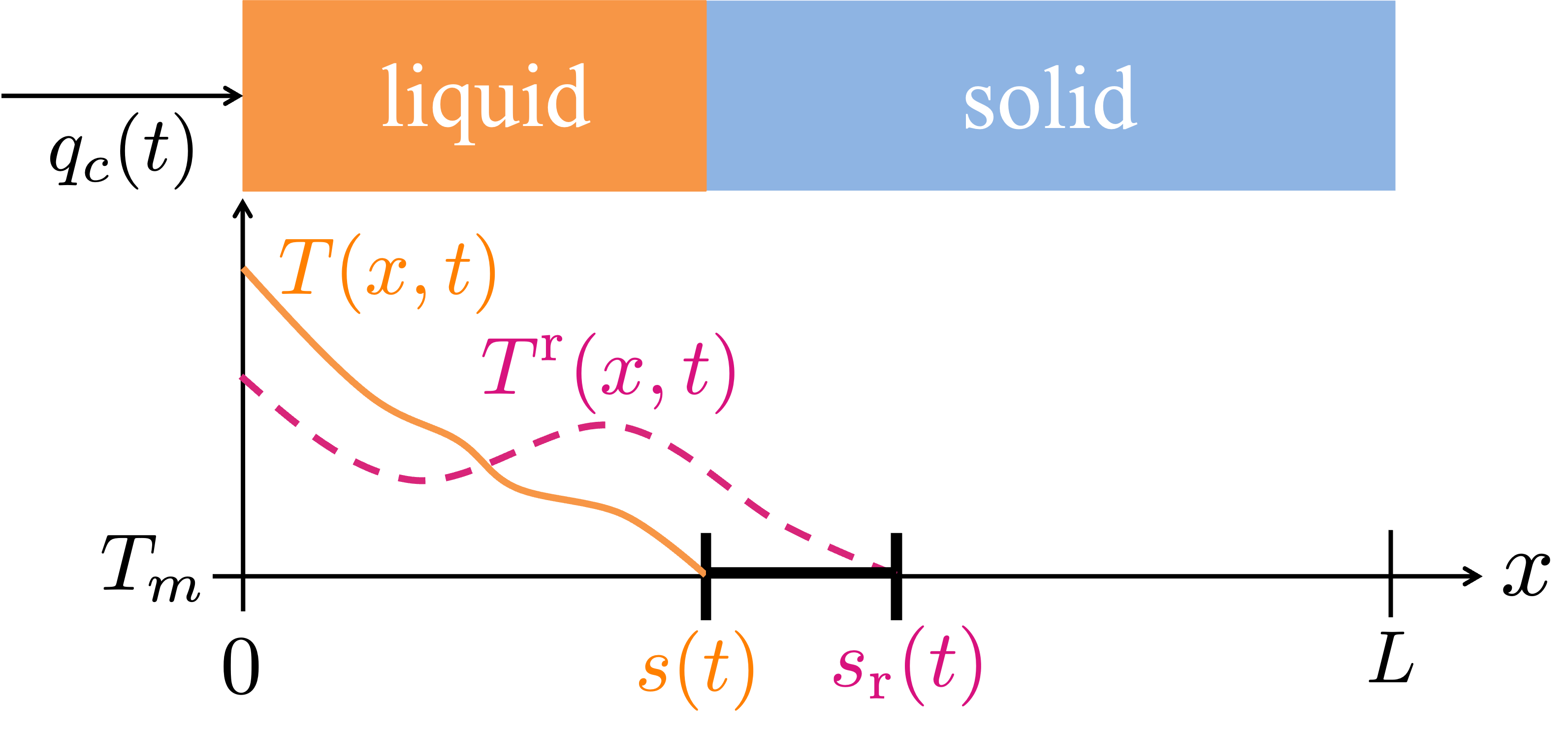}\\
\caption{Schematic of Stefan problem and a reference trajectory.}
\label{fig:stefan}
\end{figure}

\subsection{Stefan Problem with Second-Order Moving Boundary}


We consider the one-phase Stefan problem modeling the dynamics of the temperature profile $T(x,t)$ defined on the time-varying domain $x \in [0, s(t)]$ at time $t \geq 0$, described by  (see  Fig.~\ref{fig:stefan} for schematic interpretation)
\begin{align}\label{eq:stefanPDE}
T_t(x,t)&=\alpha T_{xx}(x,t), \hspace{2mm}  \textrm{for} \hspace{2mm} t > 0, \hspace{2mm} 0< x< s(t), \\ 
\label{eq:stefancontrol}
-k T_x(0,t)&=q_{\rm c}(t),  \hspace{2mm} \textrm{for} \hspace{2mm} t >0,\\ \label{eq:stefanBC}
T(s(t),t)&= T_{\rm m}, \hspace{2mm} \textrm{for} \hspace{2mm} t >0, 
\end{align}
with initial conditions 
$
 s(0) =  s_0$ and $T(x,0) = T_0(x), \forall x \in (0, s_0], $
where $\alpha>0$, $k>0$, $T_{\rm m}$ denote the diffusion coefficient, thermal conductivity, and the melting temperature, respectively. 
The classical one-phase Stefan problem, studied widely in literature including \cite{Shumon19journal}, yields the dynamics of the moving boundary as a first-order response with respect to the heat flux at the moving boundary.
Following \cite{koga2025cdc}, we consider the second-order response represented by 
\begin{align}
\label{eq:stefanODE}
\varepsilon \ddot s(t) & = - \dot s(t) - \beta T_x(s(t),t) , 
\end{align}
and $
\dot s(0) = v_0$,
where $\varepsilon>0$ denotes the relaxation time 
in phase transformation and $\varepsilon \ddot s(t)$ represents the inertia of thermal process. 

The validity of the Stefan model \eqref{eq:stefanPDE}-\eqref{eq:stefanODE} are justified only when the following two conditions are met:  
\begin{align}\label{eq:temp-valid}
T(x,t) \geq& T_{{\rm m}}, \quad  \forall x\in(0,s(t)), \quad \forall t>0, \\
\label{eq:interface-valid}0 < s(t)<  &L, \quad \forall t>0. 
\end{align}
The first condition implies that the liquid phase is not frozen from the heat inlet. The second condition implies that the material is not entirely in one phase. These physical conditions are also required for the existence and uniqueness of solutions under the following assumptions. 

\begin{assum}\label{ass:initial} 
$0 < s_0 < L$, $T_0(x) \in C^1([0, s_0];[T_{\rm m}, +\infty))$ with $T_0(s_0) = T_{\rm m}$.
 \end{assum}
 \begin{assum}\label{ass:CBFinitial}
     $v_0 \geq 0$. 
 \end{assum}

We prove the following lemma. 
 \begin{lemma}\label{lem1}
With Assumptions \ref{ass:initial}-\ref{ass:CBFinitial}, if $q_{\rm c}(t)$ is a bounded piecewise continuous non-negative heat function, i.e.,
\begin{align} \label{eq:qct-valid}
q_{{\rm c}}(t) \geq 0,  \quad \forall t\geq 0,  
\end{align}
then there exists a unique classical solution for the Stefan problem \eqref{eq:stefanPDE}--\eqref{eq:stefanODE}, which satisfies \eqref{eq:temp-valid}, and 
\begin{align} \label{eq:sdot-pos} 
    \dot s(t) \geq 0, \quad \forall t \geq 0. 
\end{align}

\end{lemma}

\begin{proof}
By applying the maximum principle for parabolic PDEs and Hopf's lemma, it holds that $T_{x}(s(t),t) \leq 0$, which leads to 
 $   \varepsilon \ddot s(t) \geq - \dot s(t). $
Applying the Gronwall's inequality yields $
    \dot s(t) \geq v_0 e^{ - \varepsilon^{-1} t}$. 
Thus, with Assumption \ref{ass:CBFinitial}, \eqref{eq:sdot-pos} holds

\end{proof}

\subsection{Safe Trajectory Tracking}

The objective is to develop a boundary control $q_{\rm c}(t)$ acting in \eqref{eq:stefancontrol} to track a given reference trajectory $s_{\rm r}(t)$ in the moving boundary. As the boundary control is fed into the temperature dynamics, the reference temperature profile $T^{\rm r}(x,t)$ needs to be solved, which must satisfy
\begin{align}\label{eq:reference_PDE}
T^{\rm r}_t(x,t)&=\alpha T^{\rm r}_{xx}(x,t), \hspace{1mm}  \textrm{for} \hspace{1mm} t > 0, \hspace{1mm} 0< x< s_{\rm r}(t), \\ 
\label{eq:reference_BC1}
-k T^{\rm r}_x(0,t)&=q^r_{\rm c}(t),  \hspace{2mm} \textrm{for} \hspace{2mm} t >0,\\ \label{eq:reference_BC2}
T^{\rm r}(s_{\rm r}(t),t)&= T_{\rm m}, \hspace{2mm} \textrm{for} \hspace{2mm} t >0, \\
\label{eq:reference_ODE} \varepsilon \ddot s_{\rm r}(t) & = - \dot s_{\rm r}(t) - \beta T^{\rm r}_x(s_{\rm r}(t),t) , 
\end{align}
where $q^r_{\rm c}(t)$ serves as a feedforward control input to be designed. The safe trajectorty tracking control problem we tackle in this paper is addressed below. 

\textbf{Problem:} Given states $(T(x,t), s(t))$ governed by \eqref{eq:stefanPDE}--\eqref{eq:stefanODE} and the reference trajectory $s_{\rm r}(t)$, 
design a safe trajectory tracking controller for $q_{\rm c}(t)$ 
so that the closed-loop system satisfies the constraint \eqref{eq:temp-valid} and $ (T(x,t), s(t)) \to (T^{\rm r}(x,t), s_{\rm r}(t))$ as a reference solution to \eqref{eq:reference_PDE}--\eqref{eq:reference_ODE}.

\section{CONTROL DESIGN} \label{sec:methods}

\subsection{Feedforward Control}

Obtaining the temperature profile solution $T^{\rm r}(x,t)$ for a known moving boundary $s_{\rm r}(t)$ is known as the inverse Stefan problem~\cite{gol2012inverse}. Once the solution to the inverse Stefan problem is obtained, the feedforward control $q_{\rm c}^r(t)$ can be derived from the boundary condition in \eqref{eq:reference_BC1}. Among several approaches, we employ the series expansion method in \cite{dunbar2003motion}. Namely, we formulate the reference temperature profile as
\begin{align} \label{eq:feedforward_series}
    T^{\rm r}(x,t) = T_{\rm m} + \sum_{n = 0}^{\infty} \frac{a_n(t)}{n !} (x - s_{\rm r}(t))^n,
\end{align}
where $a_n(t)$ for $n = 0, 1, \dots$ are time-varying coefficients to be determined. Substituting \eqref{eq:feedforward_series} into the boundary conditions \eqref{eq:reference_BC2} and \eqref{eq:reference_ODE} provides the following initial coefficients:
\begin{align} \label{eq:series_initial}
    a_0(t) = 0, \quad a_1(t) = -\frac{1}{\beta} (\varepsilon \ddot s_{\rm r}(t) + \dot s_{\rm r}(t)).
\end{align}
Moreover, substituting \eqref{eq:feedforward_series} into the reference PDE \eqref{eq:reference_PDE} gives a recursive update for $a_n(t)$ where $n \geq 2$:
\begin{align} \label{eq:series_recursion}
 a_{n}(t) = \frac{1}{\alpha} (\dot a_{n-2}(t) - \dot s_{\rm r}(t) a_{n-1}(t) ).
\end{align}
Similarly to \cite{dunbar2003motion}, to ensure the convergence of the series solution, we impose the following assumption.

\begin{assum} \label{ass:gevrey}
The reference trajectory $s_{\rm r}(t)$ is Gevrey of order $(d-1)$ for some $d \in [1, 2]$, i.e., there exist constants $M>0$ and $R>0$ such that the following bound holds
\begin{align} \label{eq:reference_bound}
 \big| s_{\rm r}^{(m+1)}(t) \big| \leq M \frac{m!^d}{R^m}
\end{align}
for all $m \in \mathbb{Z}_+$, where $s_{\rm r}^{(m)}(t) := \frac{{\rm d}^m}{{\rm d} t^m} s_{\rm r}(t)$.
\end{assum}

We prove the following proposition.
\begin{proposition} \label{prop:series}
With Assumption \ref{ass:gevrey}, it holds
\begin{align} \label{eq:series_bound}
    \big|a_n^{(m)}(t) \big | \leq M F^{n-1} G  H_{n, m}
\end{align}
where $G := \frac{1}{\beta}\left( \varepsilon + R  \right)$, $
 H_{n,m}  = \frac{ 1}{R^{n+m}} \frac{ ((n + m) !)^d }{(n !)^{d-1} }$, and
\begin{align}
 F & = \frac{ R M + \sqrt{R^2 M^2 + 16 \alpha R }}{4 \alpha}. \label{eq:series_A_def} \end{align} 
\end{proposition}
\begin{proof}
We prove \eqref{eq:series_bound} holds for all $m$ by induction on $n = 0, 1, \dots$. The detailed steps of this proof follow the methodology in \cite{dunbar2003motion}. For the base case $n=0$, the inequality holds for all $m$ since $a_0(t) = 0$ from \eqref{eq:series_initial}. For $n=1$, taking the $m$-th derivative of the second condition of \eqref{eq:series_initial} and applying \eqref{eq:reference_bound} leads to
\begin{align}
    | a_1^{(m)}(t) | 
    & \leq M G \frac{ (m+1)!^d}{ R^{m+1}},
\end{align}
thus, \eqref{eq:series_bound} holds for $n=1$.

Next, we prove \eqref{eq:series_bound} for $n\geq 2$, assuming as the induction hypothesis that it holds for $n-1$ and $n-2$. Taking the bound of \eqref{eq:series_recursion} and applying the triangle inequality yields
\begin{align}\label{eq:series_bound_step1}
    \big|a_n^{(m)}(t) \big | \leq \frac{1}{\alpha} \left( \big|  a_{n-2}^{(m+1)}(t) \big | + \big |(\dot s_{\rm r} a_{n-1} )^{(m)} \big| \right).
\end{align}
Following the derivations in \cite{dunbar2003motion} by applying the induction hypothesis to the terms on the right-hand side, we derive the following bound:
\begin{align}
   \big|a_n^{(m)}(t) \big  | \leq & M F^{n-3} G H_{n,m} \frac{R}{\alpha} \left( \left( 1 - \frac{1}{n} \right) +  \frac{ M F }{ n} \right).
\end{align}
Finally, by choosing $F$ such that the final factor equals 1, we obtain a quadratic equation for $F$. The unique positive solution is given by \eqref{eq:series_A_def}, which completes the proof.
\end{proof}

\begin{proposition} \label{prop:reference_bound}
The radius of convergence of the infinite series \eqref{eq:feedforward_series} is given by $R / F$. Namely, if
\begin{align} \label{eq:convergence}
    R \geq F \sup_{t \geq 0} |s_{\rm r}(t)|,
\end{align}
then the series \eqref{eq:feedforward_series} converges, and it holds:
\begin{align} \label{eq:Tr_bound}
     | T^{\rm r}(x,t) - T_{\rm m} | & \leq  \frac{M G R}{F \left( R  - F (s_{\rm r}(t) - x)\right)}, \\
     \label{eq:Trx_bound}  |  T^{\rm r}_x(x,t) | & \leq \frac{MGR }{(R - F(s_{\rm r}(t) - x))^2}.
\end{align}
\end{proposition}
\begin{proof}
The bounds are found by applying the triangle inequality to the series \eqref{eq:feedforward_series} and its spatial derivative, followed by applying the coefficient bound from \eqref{eq:series_bound}. For the first inequality, we have:
\begin{align}
    | T^{\rm r}(x,t) - T_{\rm m} | 
    \leq M G F^{-1} \sum_{n = 0}^{\infty} \left( \frac{ F} {R} ( s_{\rm r}(t) - x) \right)^n.
\end{align}
Under the condition in \eqref{eq:convergence}, the term $\frac{F}{R}(s_{\rm r}(t)-x)$ is less than 1, leading directly to the bound in \eqref{eq:Tr_bound}. The bound \eqref{eq:Trx_bound} is derived similarly.
\end{proof}

Applying the series solution of the reference temperature \eqref{eq:feedforward_series} to the boundary condition \eqref{eq:reference_BC1}, the feedforward control $q^r_{\rm c}(t)$ is obtained analytically as
\begin{align} \label{eq:feedforward}
    q^r_{\rm c} (t) = - k \sum_{n = 0}^{\infty} \frac{a_{n+1}(t)}{n !} ( - s_{\rm r}(t))^n.
\end{align}
The inverse Stefan problem is known to be ill-posed; ensuring that the classical solution $T^{\rm r}(x,t)$ satisfies the physical constraint \eqref{eq:temp-valid} is generally not guaranteed. Hence, we impose the following assumptions to ensure the well-posedness of the reference state.

\begin{assum} \label{ass:reference_trajectory}
The reference trajectory $s_{\rm r}(t)$ satisfies
\begin{align}
    s_{\rm r}(0) >0, \quad \dot s_{\rm r}(t) \geq 0, \quad \lim_{t \to \infty} s_{\rm r}(t) = \bar s_{\rm r} < L.
\end{align}
\end{assum}

\begin{assum} \label{ass:reference_temp}
The series solution \eqref{eq:feedforward_series} satisfies \eqref{eq:temp-valid}.
\end{assum}

\subsection{Energy-Shaping Tracking Control with Safety Guarantee}

Let the total energy be defined as
\begin{align} \label{eq:energy-def}
    E(t) = \int_0^{s(t)} (T(x,t) - T_{\rm m}) dx + \frac{k}{\beta} (\varepsilon \dot s(t) + s(t)).
\end{align}
Taking the time derivative of \eqref{eq:energy-def} along the system's solution yields the energy conservation law $
    \dot E(t) = q_{\rm c}(t)$.
Let $E_{\rm r}(t)$ be the reference energy defined by 
\begin{align} \label{eq:energy_ref-def}
 \hspace{-2.2mm}   E_{\rm r}(t) = &\int_0^{s_{\rm r}(t)} (T^{\rm r}(x,t) - T_{\rm m}) dx + \frac{k}{\beta} (\varepsilon \dot s_{\rm r}(t) + s_{\rm r}(t)).
\end{align}
We design an energy-shaping controller as
\begin{align} \label{eq:tracking_control}
    q_{\rm c}(t) = q_{\rm c}^{\rm r}(t) - c \left( E(t) - E_{\rm r}(t) \right),
\end{align}
where $c>0$ is a control gain and $q_{\rm c}^{\rm r}(t)$ is the feedforward control from \eqref{eq:feedforward}. 
To ensure safety, we impose the following assumption.

\begin{assum} \label{ass:energy}
The initial energy satisfies
\begin{align}
 \label{eq:energy_cond_1}  E(0) \leq E_{\rm r}(0) +  \frac{1}{c} q_{\rm c}^{\rm r}(t) e^{ct}, \quad \forall t \geq 0, \\
\label{eq:energy_cond_2}   E_{\rm r}(t) + (E(0) - E_{\rm r}(0)) e^{-ct} < \frac{k}{\beta} \bar s_{\rm r}, \quad \forall t \geq 0.
\end{align}
\end{assum}

\begin{proposition} \label{prop:safety}
Let Assumptions \ref{ass:initial}--\ref{ass:energy} hold. Then the closed-loop system, consisting of the Stefan dynamics \eqref{eq:stefanPDE}--\eqref{eq:stefanODE}, the tracking control \eqref{eq:tracking_control}, and the feedforward control \eqref{eq:feedforward}, satisfies the safety condition \eqref{eq:temp-valid} and
\begin{align} \label{eq:int-valid}
    s_0 < s(t) \leq \bar s_{\rm r}, \quad \dot s(t) \geq 0, \quad \forall t \geq 0.
\end{align}
\end{proposition}
\begin{proof}
Substituting the controller \eqref{eq:tracking_control} into the energy conservation law yields a first-order linear ODE for the energy error $E(t) - E_{\rm r}(t)$, whose solution is
\begin{align} \label{eq:energy_decay}
    E(t) - E_{\rm r}(t) = (E(0) - E_{\rm r}(0)) e^{-ct}.
\end{align}
From Lemma \ref{lem1}, \eqref{eq:temp-valid} is guaranteed if $q_{\rm c}(t) \geq 0$. Using \eqref{eq:tracking_control} and \eqref{eq:energy_decay}, this safety requirement becomes $q_{\rm c}^{\rm r}(t) - c(E(0)-E_{\rm r}(0))e^{-ct} \geq 0$, which gives \eqref{eq:energy_cond_1}. Furthermore, the interface constraint $s(t) \leq \bar{s}_{\rm r}$ and non-negativity of the temperature profile require $E(t) \leq \frac{k}{\beta} \bar s_{\rm r}$. Substituting $E(t)$ from \eqref{eq:energy_decay} into this inequality yields \eqref{eq:energy_cond_2}.
\end{proof}

\section{Stability Analysis}

This section presents the main theoretical result: the exponential stability of the closed-loop trajectory tracking error. 

\begin{theorem} \label{thm:stability}
Let Assumptions \ref{ass:initial}--\ref{ass:energy} hold. Consider the closed-loop system of \eqref{eq:stefanPDE}--\eqref{eq:stefanODE} with the tracking control \eqref{eq:tracking_control}. Then, there exist positive constants $c^*, \bar M, b$ such that for any controller gain $c \geq c^*$, it holds $\Phi(t) \leq \bar M \Phi(0) e^{- bt}$, where $
\Phi(t) := \int_0^{s(t)} \big(T(x,t) - T^r(x - (s(t) - s_{\rm r}(t)),t)\big)^2 dx + (s(t) - s_{\rm r}(t))^2+ ( \dot s(t) - \dot s_{\rm r}(t))^2$.
\end{theorem}

Note that the integral term in $\Phi(t)$ measures the error between the system temperature and a reference profile that is spatially shifted to align with the current interface position $s(t)$. The remainder is dedicated to the proof of this theorem.

\subsection{Reference Error and Target Systems}

Let the tracking error states be defined as $
\tilde s(t) := s(t) - s_{\rm r}(t), \hspace{1mm} X(t) := \begin{bmatrix} \tilde s(t) \dot{\tilde{s}}(t) \end{bmatrix}^\top, \hspace{1mm}
u(x,t) := T(x,t) - T^r(x - \tilde s(t), t), \hspace{1mm} \forall x \in (0, s(t))$.
Subtracting the reference system \eqref{eq:reference_PDE}--\eqref{eq:reference_ODE} from the Stefan system \eqref{eq:stefanPDE}--\eqref{eq:stefanODE} yields:
\begin{align} \label{eq:u-PDE}
u_t &= \alpha u_{xx} + e_2^\top X T^r_x(x- e_1^\top X,t), \\
\label{eq:u-BC1} u_x(0,t) & = - k^{-1} q_{\rm c}(t) - T^r_x(- e_1^\top X, t), \\
\label{eq:u-BC2} u(s(t),t) & = 0, \\
\label{eq:u-ODE} \dot X(t) & = AX(t) + B u_x(s(t),t),
\end{align}
where $e_1 = [1, 0]^\top, e_2 = [0, 1]^\top$, and the matrices $A$ and $B$ are defined by $
    A = \left [
    \begin{array}{cc}
       0  & 1 \\
       0  & - \varepsilon^{-1}
    \end{array}
    \right] , \quad 
    B =\left [
    \begin{array}{c}
       0  \\
       - \varepsilon^{-1} \beta
    \end{array}
    \right]  . 
$

We now introduce the following backstepping transformation to map the error state $u(x,t)$ to a target state $w(x,t)$:
$w(x,t)  = u(x,t) +  \frac{1}{\alpha}  \int_x^{s(t)} \phi(x-y)^\top B u(y,t) dy
 - \phi^\top (x-s(t)) X(t)$, 
where 
$\phi(x) = Kx - D,
K^\top = \frac{c}{\beta} \begin{bmatrix} 1 & \varepsilon \end{bmatrix},$
and $D \in \bbR^2$ is to be determined. 
Note that $A + BK$ is a Hurwitz matrix for any control gain $c>0$. Applying the transformation to the reference error system \eqref{eq:u-PDE}--\eqref{eq:u-ODE} results in the target system for the $(w,X)$ states, which is governed by the following dynamics:
\begin{align}
w_t &= \alpha w_{xx} + D^\top A X(t) + e_2^\top X I(T^r_x(x- e_1^\top X,t)) \notag\\
& \quad + \dot s(t) K^\top X(t), \label{eq:target_PDE} \\
w_x(0,t) & = \frac{D^\top B}{\alpha} g(w(x,t), X(t)) + d(e_1^\top X), \label{eq:target_BC1} \\
w(s(t),t) & = D^\top X(t), \label{eq:target_BC2}\\
\dot X(t) & = (A + BK) X(t) + B w_x(s(t),t). \label{eq:target_ODE}
\end{align}
and $d$, $I$ are defined by $I(f(x,t)) = f(x,t) + \frac{1}{\alpha}  \int_x^{s(t)} \phi(x-y)^\top B f(y,t) dy, $ $
d(e_1^\top X) = - k^{-1} q_{\rm c}^r(t) - T^r_x(- e_1^\top X, t)$, 
and $g(w(x,t), X(t))$ is derived via the inverse transformation provided in \cite{koga2021towards} (see eq. (35)-(37) therein). 
Substituting the feedforward input \eqref{eq:feedforward} and the reference solution \eqref{eq:feedforward_series} into $d(e_1^\top X)$ 
and applying the results of Proposition \eqref{prop:reference_bound} leads to $
    | d(e_1^\top X) | 
   \leq \frac{2 M F G R}{ \left(R  -  F \bar s_{\rm r}  \right)^3} | X |  $.

\subsection{Lyapunov Analysis}

We consider the following Lyapunov function candidate $
V(t) = \frac{1}{2} \int_0^{s(t)} w(x,t)^2 dx  +  X(t)^\top P X(t) $,
where $P \in \mathbb{R}^{2 \times 2}$ is a symmetric positive definite matrix satisfying the Lyapunov equation $
P(A + BK)  + (A + BK)^\top P = - Q$,
for a given positive semidefinite matrix $Q = \lambda \diag(1, \kappa)$ with $\lambda, \kappa \geq 0$. 
Taking the time derivative of $V$ along with \eqref{eq:target_PDE}--\eqref{eq:target_ODE}
owns a cross-term $X w_x(s(t),t)$, which can be canceled by choosing the vector $D$ as $
D = - \frac{2}{\alpha} P B $.
As the solution to the Lyapunov equation can be explicitly obtained with the matrix $Q = \lambda \diag(1, \kappa)$ and $(A, B, K)$, an explicit form of $D$ can be given as $
    D = \lambda \beta \left [
    \begin{array}{cc}
        \frac{1}{ c}  &
     \frac{\varepsilon}{ c (1 + c \varepsilon)} + \frac{\kappa}{ 1 + c \varepsilon}
    \end{array}
    \right]^\top  $.  
Applying Young's, Cauchy-Schwarz, and Poincaré's inequalities to all the terms except for $ \int_0^{s(t)} w_x(x,t)^2 dx$ and $X^\top Q X$ in the time derivative of the Lyapunov function, one can show that there exist positive constants $M_i$ for $i \in \{1, 2, 3, 4\}$ that are independent on parameters $(\lambda, c)$ such that $
\dot V \leq  - \alpha \left(\frac{1}{4} - M_1 |D| \right) \int_0^{s(t)} w_x(x,t)^2 dx - (\lambda - M_2 - M_3 |D| ) |X|^2 + M_4 \dot s(t) V$ holds.
Therefore, by first choosing $\lambda \geq 4 M_2$ and then selecting a sufficiently large gain $c \geq c^*$ to make $|D|$ small enough satisfying $|D| < \min \left\{ \frac{1}{8M_1},  \frac{M_2}{M_3} \right\}$, we can ensure that the first two terms in the inequality of $\dot V$ are negative definite. This guarantees the existence of a constant $M_5>0$ such that $
\dot V \leq - M_5 V + M_4 \dot s(t) V$.
The final term $M_4 \dot s(t) V$ in is handled by applying the conditions $\dot s(t) \geq 0$ and $s_0<s(t) < L$ shown in Proposition \ref{prop:safety}, similarly to the procedure established in \cite{Shumon19journal}, which leads to the exponential decay of $V(t)$. By norm equivalence between $(w,X)$ and $(u,X)$, the exponential decay of $\Phi(t)$ is proven, which completes the proof of Theorem \ref{thm:stability}.

\section{SIMULATION} \label{sec:results}

\begin{figure*}[t]
\centering 
\subfloat[The interface position.]
{\includegraphics[width=0.24 \linewidth]{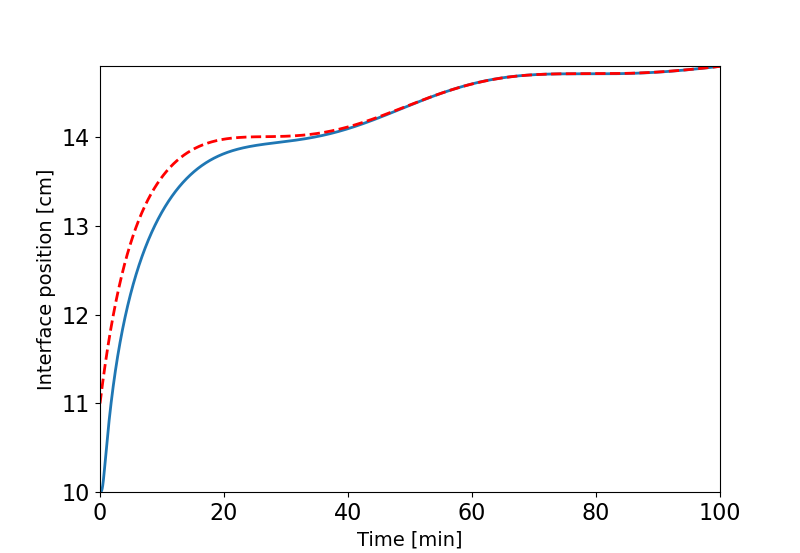}\label{fig:interface}} 
\subfloat[The boundary heat flux input. ]
{\includegraphics[width=0.24 \linewidth]{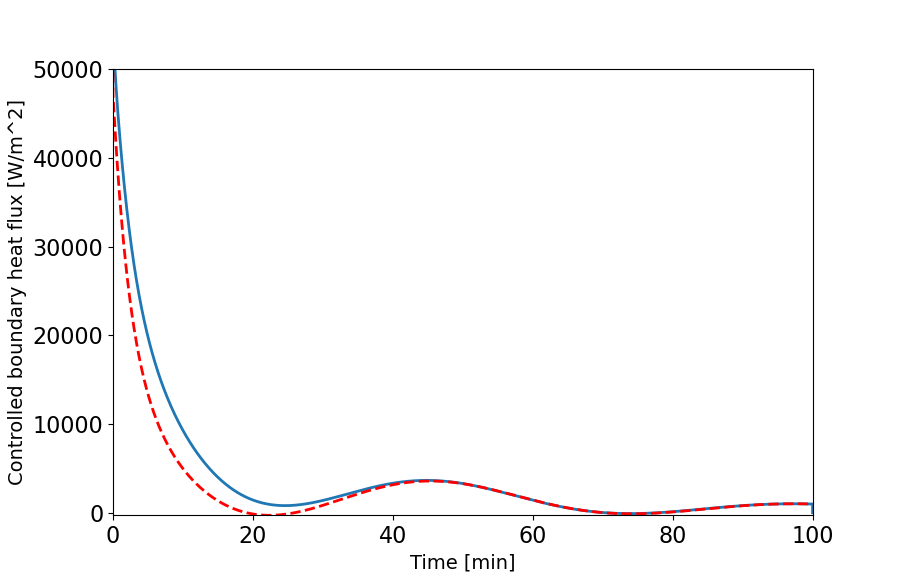}\label{fig:qc}} 
\hspace{1mm}
\subfloat[The temperature surface plots of tracking control and reference.]
{\includegraphics[width=0.5 \linewidth ]{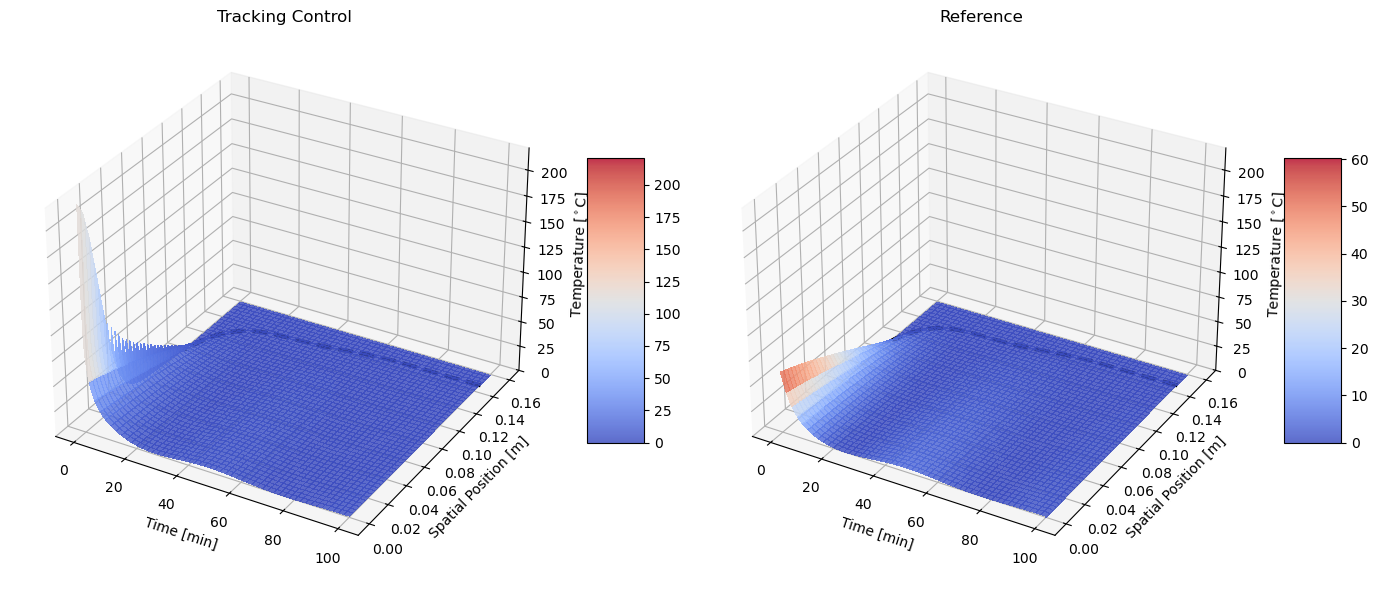}\label{fig:T0}}\\
\subfloat[The snapshots of the temperature profile at chosen time.]
{\includegraphics[width=\linewidth]{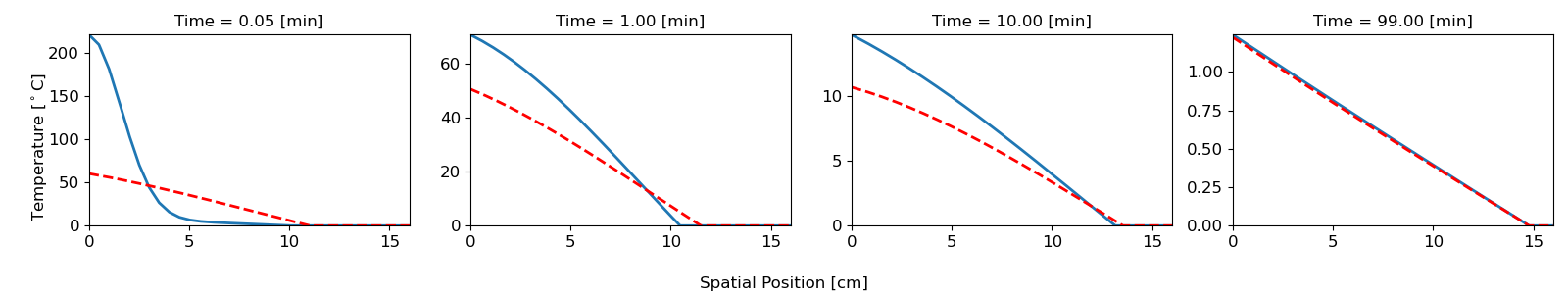}\label{fig:snapshots}} 
\caption{ The closed-loop response of \eqref{eq:stefanPDE}--\eqref{eq:stefanODE} with the trajectory tracking control \eqref{eq:tracking_control}, \eqref{eq:energy-def}, \eqref{eq:feedforward} with the reference trajectory \eqref{eq:sim_sref}. The plots are given for tracking control (blue solid) and its reference (red dash). }
\label{fig:response}
\end{figure*}

This section studies numerical simulation. For physical constants, we use the parameters for Zinc from \cite{koga2025cdc}, and the melting temperature is set to $T_{\rm m} = 0$ [°C] without loss of generality. 
The other parameters and initial conditions are set to $\varepsilon = 10 $ [s],  $s_0 = 0.1$ [m], $T_{0}(x) = 10 (1 - x / s_0)$ [°C], and $c = 0.002$ [1/s]. 
For the reference trajectory, we set
\begin{align}\label{eq:sim_vref}
    \dot s_{\rm r}(t) & = A (1 + \cos(\omega t)) e^{- \delta_1 t} + v_{\rm min} e^{- \delta_2 t},
\end{align}
where $A>0$ is to be determined, and $\omega, \delta_1, \delta_2, v_{\rm min}$ are user-chosen positive constants. This form is motivated by the need for a non-trivial sinusoidal behavior while ensuring that Assumption \ref{ass:reference_trajectory} holds (i.e., $\dot s_{\rm r}(t) \geq 0$ and $\lim_{t \to \infty} \dot s_{\rm r}(t) = 0$). Moreover, adding $v_{\rm min} e^{- \delta_2 t}$ helps prevent the feedforward control from demanding excessive cooling as $\dot s_{\rm r}(t)$ approaches zero. Integrating \eqref{eq:sim_vref} with respect to time provides 
\begin{align}
s_{\rm r}(t) & = A  \left( \frac{\omega \sin(\omega t) e^{-\delta_1 t} - \delta_1 (\cos(\omega t) e^{-\delta_1 t} - 1)}{\delta_1^2 + \omega^2} \right) \notag\\
& \quad + \frac{A}{\delta_1} (1 - e^{-\delta_1 t})  + \frac{v_{\rm min}}{\delta_2} ( 1 - e^{- \delta_2 t}) + s_{\rm r}(0). \label{eq:sim_sref}
\end{align}
By taking the limit $t \to \infty$ and setting $s_r(\infty) = \bar{s}_r$, we have $
    A = \frac{\bar s_{\rm r} - s_{\rm r}(0) - \frac{v_{\rm min}}{\delta_2}}{\frac{\delta_1}{\delta_1^2 + \omega^2} + \frac{1}{\delta_1}}$.
The simulation parameters are chosen to ensure $A>0$, with specific values
$\omega = 0.002$ [rad/s], $v_{\rm min} = 7.0 \times 10^{-7}$ [m/s], $\delta_1 = 4.0 \times 10^{-4}$ [/s], $\delta_2 = 4.0 \times 10^{-3}$ [/s], $s_{\rm r}(0) = 0.11$ [m], and $\bar s_{\rm r} = 0.15$ [m]. 

The simulation results are shown in Fig. \ref{fig:response}. In all plots, the response of the closed-loop system is shown as a blue solid line, while the reference is a red dashed line. Fig. \ref{fig:interface} shows the temporal evolution of the interface position. The reference trajectory increases monotonically but with a varying growth rate due to the sinusoidal term. The controlled interface position successfully converges to the reference trajectory within 60 minutes.
Figs. \ref{fig:qc} and \ref{fig:T0} provide the temporal plot of the boundary heat flux and a 3D surface plot of the spatiotemporal temperature profile, respectively. As shown in Fig. \ref{fig:qc}, the reference heat flux is non-monotonic and even becomes slightly negative at some time, which corresponds physically to cooling. In contrast, the actual heat flux applied by the tracking controller (blue solid line) remains non-negative, a direct consequence of the safety guarantees provided by Assumption \ref{ass:energy}, even while tracking the non-monotonic reference. Despite the negative reference input, the reference temperature profile shown in Fig. \ref{fig:T0} always maintains non-negative values, satisfying the condition in Assumption \ref{ass:reference_temp}. The closed-loop system also satisfies the safety constraint \eqref{eq:temp-valid}, with the temperature remaining above the melting point for all time.
Finally, Fig. \ref{fig:snapshots} displays snapshots of the spatial temperature profile at selected time instances ($t = 0.05, 1, 10, 99$ [min]). Because the initial internal energy of the system is lower than that of the reference, the controller initially applies a large heat flux. This causes the temperature to rise rapidly, briefly overshooting the reference profile right after the initial time, here depicted at $t=0.05$ [min]. By $t=1$ [min], the controller has reduced the heat flux, and the temperature profile is much closer to the reference. As time progresses, the tracking error for both temperature and interface position continues to decrease. At the final time of $t=99$ [min], the system state is nearly identical to the reference state. Overall, the results in Fig. \ref{fig:response} are consistent with the theoretical results proven in this paper and demonstrate successful tracking performance for a non-trivial sinusoidal reference trajectory.





\bibliographystyle{cls/IEEEtran.bst}
\bibliography{bib/IEEEabrv.bib,bib/ref.bib}

@article{srisuma2025simulation,
  title={Simulation-based approach for fast optimal control of a Stefan problem with application to cell therapy},
  author={Srisuma, Prakitr and Barbastathis, George and Braatz, Richard D},
  journal={Automatica},
  volume={179},
  pages={112398},
  year={2025},
  publisher={Elsevier}
}

@article{belhamadia2023numerical,
  title={Numerical modelling of hyperbolic phase change problems: Application to continuous casting},
  author={Belhamadia, Youssef and Cassol, Guilherme Ozorio and Dubljevic, Stevan},
  journal={International Journal of Heat and Mass Transfer},
  volume={209},
  pages={124042},
  year={2023},
  publisher={Elsevier}
}

@article{francis2025spatial,
  title={Spatial modeling algorithms for reactions and transport in biological cells},
  author={Francis, Emmet A and Laughlin, Justin G and Dokken, J{\o}rgen S and Finsberg, Henrik NT and Lee, Christopher T and Rognes, Marie E and Rangamani, Padmini},
  journal={Nature Computational Science},
  volume={5},
  number={1},
  pages={76--89},
  year={2025},
  publisher={Nature Publishing Group US New York}
}

@article{nwankwo2024deep,
  title={Deep learning and American options via free boundary framework},
  author={Nwankwo, Chinonso and Umeorah, Nneka and Ware, Tony and Dai, Weizhong},
  journal={Computational Economics},
  volume={64},
  number={2},
  pages={979--1022},
  year={2024},
  publisher={Springer}
}

@article{petrus2022solid,
  title={Solid boundary output feedback control of the stefan problem: The enthalpy approach},
  author={Petrus, Bryan and Chen, Zhelin and El-Kebir, Hamza and Bentsman, Joseph and Thomas, Brian G},
  journal={IEEE transactions on automatic control},
  volume={68},
  number={6},
  pages={3485--3500},
  year={2022},
  publisher={IEEE}
}

@article{koga2020energy,
  title={Energy storage in paraffin: A {PDE} backstepping experiment},
  author={Koga, Shumon and Makihata, Mitsutoshi and Chen, Renkun and Krstic, Miroslav and Pisano, Albert P},
  journal={IEEE Transactions on Control Systems Technology},
  volume={29},
  number={4},
  pages={1490--1502},
  year={2020},
  publisher={IEEE}
}

@article{vazquez2026backstepping,
  title={Backstepping for partial differential equations: A survey},
  author={Vazquez, Rafael and Auriol, Jean and Bribiesca-Argomedo, Federico and Krstic, Miroslav},
  journal={Automatica},
  volume={183},
  pages={112572},
  year={2026},
  publisher={Elsevier}
}

@article{karafyllis2022spill,
  title={Spill-free transfer and stabilization of viscous liquid},
  author={Karafyllis, Iasson and Krstic, Miroslav},
  journal={IEEE Transactions on Automatic Control},
  volume={67},
  number={9},
  pages={4585--4597},
  year={2022},
  publisher={IEEE}
}

@article{roy2024input,
  title={An Input-to-State Safety Approach Toward Safe Control of a Class of Parabolic {PDEs} Under Disturbances},
  author={Roy, Tanushree and Knichel, Ashley and Dey, Satadru},
  journal={IEEE Transactions on Control Systems Technology},
  volume={32},
  number={5},
  pages={1936--1943},
  year={2024},
  publisher={IEEE}
}

@article{bhan2023neural,
  title={Neural operators for bypassing gain and control computations in PDE backstepping},
  author={Bhan, Luke and Shi, Yuanyuan and Krstic, Miroslav},
  journal={IEEE Transactions on Automatic Control},
  volume={69},
  number={8},
  pages={5310--5325},
  year={2023},
  publisher={IEEE}
}

@article{hu2025safe,
  title={Safe PDE Boundary Control with Neural Operators},
  author={Hu, Hanjiang and Liu, Changliu},
  journal={7th Annual Learning for Dynamics \& Control Conference},
  year={2025}
}

@article{ecklebe2021model,
  title={Model predictive control of the vertical gradient freeze crystal growth process},
  author={Ecklebe, Stefan and Buchwald, Tom and R{\"u}diger, Patrick and Winkler, Jan},
  journal={IFAC-PapersOnLine},
  volume={54},
  number={6},
  pages={218--225},
  year={2021},
  publisher={Elsevier}
}

@article{ran2024heat,
  title={Heat conduction in live tissue during radiofrequency electrosurgery},
  author={Ran, Junren and El-Kebir, Hamza and Lee, Yongseok and Chamorro, Leonardo P and Berlin, Richard and Aguiluz Cornejo, Gabriela M and Benedetti, Enrico and Giulianotti, Pier C and Bhargava, Rohit and Bentsman, Joseph and others},
  journal={Journal of the Royal Society Interface},
  volume={21},
  number={210},
  pages={20230420},
  year={2024},
  publisher={The Royal Society}
}

@article{si2022coupled,
  title={Coupled ocean--sea ice dynamics of the Antarctic Slope Current driven by topographic eddy suppression and sea ice momentum redistribution},
  author={Si, Yidongfang and Stewart, Andrew L and Eisenman, Ian},
  journal={Journal of Physical Oceanography},
  volume={52},
  number={7},
  pages={1563--1589},
  year={2022}
}

@article{koga2023safe,
  title={Safe {PDE} backstepping {QP} control with high relative degree {CBFs}: Stefan model with actuator dynamics},
  author={Koga, Shumon and Krstic, Miroslav},
  journal={IEEE Transactions on Automatic Control},
  volume={68},
  number={12},
  pages={7195--7208},
  year={2023},
  publisher={IEEE}
}

@inproceedings{wang2024safe,
  title={Safe Control of Hyperbolic {PDE-ODE} Cascades},
  author={Wang, Ji and Krstic, Miroslav},
  booktitle={2024 American Control Conference (ACC)},
  pages={2533--2538},
  year={2024},
  organization={IEEE}
}

@inproceedings{koga2023event,
  title={Event-triggered safe stabilizing boundary control for the {Stefan} {PDE} system with actuator dynamics},
  author={Koga, Shumon and Demir, Cenk and Krstic, Miroslav},
  booktitle={2023 American Control Conference (ACC)},
  pages={1794--1799},
  year={2023},
  organization={IEEE}
}

@inproceedings{buisson2018control,
  title={Control of piston position in inviscid gas by bilateral boundary actuation},
  author={Buisson-Fenet, Mona and Koga, Shumon and Krstic, Miroslav},
  booktitle={2018 IEEE Conference on Decision and Control (CDC)},
  pages={5622--5627},
  year={2018},
  organization={IEEE}
}

@inproceedings{koga2025cdc,
  title={Safe Stabilization of the {Stefan} Problem with a High-Order Moving Boundary Dynamics by PDE Backstepping},
  author={Koga, Shumon and Krstic, Miroslav},
  booktitle={2025 IEEE Conference on Decision and Control (CDC)},
  pages={},
  year={2025},
  organization={IEEE}
}

@article{ecklebe2021toward,
  title={Toward Model-Based Control of the Vertical Gradient Freeze Crystal Growth Process},
  author={Ecklebe, S. and Woittennek, F. and Frank-Rotsch, C. and Dropka, N. and Winkler, J.},
  journal={IEEE Transactions on Control Systems Technology},
  volume={30},
  number={1},
  pages={384--391},
  year={2021},
  publisher={IEEE}
}

@article{agrawal2024gatekeeper,
  title={gatekeeper: Online safety verification and control for nonlinear systems in dynamic environments},
  author={Agrawal, Devansh Ramgopal and Chen, Ruichang and Panagou, Dimitra},
  journal={IEEE Transactions on Robotics},
  year={2024},
  publisher={IEEE}
}

@article{yu2020bilateral,
  title={Bilateral boundary control of moving shockwave in {LWR} model of congested traffic},
  author={Yu, Huan and Diagne, Mamadou and Zhang, Liguo and Krstic, Miroslav},
  journal={IEEE Transactions on Automatic Control},
  volume={66},
  number={3},
  pages={1429--1436},
  year={2020},
  publisher={IEEE}
}

@article{koga2021towards,
  title={Towards implementation of {PDE} control for {Stefan} system: Input-to-state stability and sampled-data design},
  author={Koga, S. and Karafyllis, I. and Krstic, M.},
  journal={Automatica},
  volume={127},
  pages={109538},
  year={2021},
  publisher={Elsevier}
}

@article{ames2016control,
  title={Control barrier function based quadratic programs for safety critical systems},
  author={Ames, A. D. and Xu, X. and Grizzle, J. W. and Tabuada, P.},
  journal={IEEE Transactions on Automatic Control},
  volume={62},
  number={8},
  pages={3861--3876},
  year={2016},
  publisher={IEEE}
}

@article{capelli2022passivity,
  title={Passivity and control barrier functions: Optimizing the use of energy},
  author={Capelli, Beatrice and Secchi, Cristian and Sabattini, Lorenzo},
  journal={IEEE Robotics and Automation Letters},
  volume={7},
  number={2},
  pages={1356--1363},
  year={2022},
  publisher={IEEE}
}

@article{ortega2002putting,
  title={Putting energy back in control},
  author={Ortega, Romeo and Van Der Schaft, Arjan J and Mareels, Iven and Maschke, Bernhard},
  journal={IEEE Control Systems Magazine},
  volume={21},
  number={2},
  pages={18--33},
  year={2002},
  publisher={IEEE}
}

@article{califano2023passivity,
  title={Passivity-preserving safety-critical control using control barrier functions},
  author={Califano, Federico},
  journal={IEEE Control Systems Letters},
  volume={7},
  pages={1742--1747},
  year={2023},
  publisher={IEEE}
}

@article{singletary2021safety,
  title={Safety-critical kinematic control of robotic systems},
  author={Singletary, Andrew and Kolathaya, Shishir and Ames, Aaron D},
  journal={IEEE Control Systems Letters},
  volume={6},
  pages={139--144},
  year={2021},
  publisher={IEEE}
}

@article{galvao2022extremum,
  title={Extremum seeking for {Stefan} {PDE} with moving boundary and delays},
  author={Galv{\~a}o, Maur{\'\i}cio Linhares and Oliveira, Tiago Roux and Krsti{\'c}, Miroslav},
  journal={IFAC-PapersOnLine},
  volume={55},
  number={36},
  pages={222--227},
  year={2022},
  publisher={Elsevier}
}

@article{yang2018output,
  title={Output consensus of multiagent systems based on {PDEs} with input constraint: A boundary control approach},
  author={Yang, Chengdong and Huang, Tingwen and Zhang, Ancai and Qiu, Jianlong and Cao, Jinde and Alsaadi, Fuad E},
  journal={IEEE Transactions on Systems, Man, and Cybernetics: Systems},
  volume={51},
  number={1},
  pages={370--377},
  year={2018},
  publisher={IEEE}
}

@article{mironchenko2020local,
  title={Local stabilization of an unstable parabolic equation via saturated controls},
  author={Mironchenko, Andrii and Prieur, Christophe and Wirth, Fabian},
  journal={IEEE Transactions on Automatic Control},
  volume={66},
  number={5},
  pages={2162--2176},
  year={2020},
  publisher={IEEE}
}

@article{dubljevic2006predictive,
  title={Predictive control of parabolic {PDEs} with state and control constraints},
  author={Dubljevic, Stevan and El-Farra, Nael H and Mhaskar, Prashant and Christofides, Panagiotis D},
  journal={International Journal of Robust and Nonlinear Control: IFAC-Affiliated Journal},
  volume={16},
  number={16},
  pages={749--772},
  year={2006},
  publisher={Wiley Online Library}
}

@article{Shumon19journal,
  title={Control and state estimation of the one-phase {Stefan} problem via backstepping design},
  author={Koga, S. and Diagne, M. and Krstic, M.},
  journal={IEEE Transactions on Automatic Control},
  volume={64},
  number={2},
  pages={510--525},
  year={2018},
  publisher={IEEE}
}

@book{KKbook2021,
  title={Materials Phase Change {PDE} Control and Estimation: From Additive Manufacturing to Polar Ice},
  author={Koga, S. and Krstic, M.},
  year={2020},
  publisher={Springer Nature}
}

@article{demir2024neuron,
  title={Neuron growth control and estimation by {PDE} backstepping},
  author={Demir, Cenk and Koga, Shumon and Krstic, Miroslav},
  journal={Automatica},
  volume={165},
  pages={111669},
  year={2024},
  publisher={Elsevier}
}

@article{wang2021deep,
  title={Deep learning of free boundary and {Stefan} problems},
  author={Wang, Sifan and Perdikaris, Paris},
  journal={Journal of Computational Physics},
  volume={428},
  pages={109914},
  year={2021},
  publisher={Elsevier}
}

@article{winter2023multi,
  title={Multi-fidelity Bayesian optimization to solve the inverse {Stefan} problem},
  author={Winter, JM and Abaidi, R and Kaiser, JWJ and Adami, S and Adams, NA},
  journal={Computer Methods in Applied Mechanics and Engineering},
  volume={410},
  pages={115946},
  year={2023},
  publisher={Elsevier}
}

@article{Hinze07,
  title={Optimal control of the free boundary in a two-phase {Stefan} problem},
  author={Hinze, Michael and Ziegenbalg, Stefan},
  journal={Journal of Computational Physics},
  volume={223},
  number={2},
  pages={657--684},
  year={2007},
  publisher={Elsevier}
}

@article{dunbar2003motion,
  title={Motion planning for a nonlinear {Stefan} problem},
  author={Dunbar, William B and Petit, Nicolas and Rouchon, Pierre and Martin, Philippe},
  journal={ESAIM: Control, Optimisation and Calculus of Variations},
  volume={9},
  pages={275--296},
  year={2003},
  publisher={EDP Sciences}
}

@article{maidi2014,
  title={Boundary geometric control of a linear stefan problem},
  author={Maidi, Ahmed and Corriou, Jean-Pierre},
  journal={Journal of Process Control},
  volume={24},
  number={6},
  pages={939--946},
  year={2014},
  publisher={Elsevier}
}

@inproceedings{petrus12,
  title={Enthalpy-based feedback control algorithms for the {Stefan} problem},
  author={Petrus, Bryan and Bentsman, Joseph and Thomas, Brian G},
  booktitle={2012 IEEE 51st IEEE Conference on Decision and Control (CDC)},
  pages={7037--7042},
  year={2012},
  organization={IEEE}
}

@article{wang2021closed,
  title={Closed-loop high-fidelity simulation integrating finite element modeling with feedback controls in additive manufacturing},
  author={Wang, Dan and Chen, Xu},
  journal={Journal of Dynamic Systems, Measurement, and Control},
  volume={143},
  number={2},
  pages={021006},
  year={2021},
  publisher={American Society of Mechanical Engineers}
}

@article{pozzato2024accelerating,
  title={Accelerating the transition to cobalt-free batteries: a hybrid model for {LiFePO4}/graphite chemistry},
  author={Pozzato, Gabriele and Li, Xueyan and Lee, Donghoon and Ko, Johan and Onori, Simona},
  journal={npj Computational Materials},
  volume={10},
  number={1},
  pages={14},
  year={2024},
  publisher={Nature Publishing Group UK London}
}

@book{gol2012inverse,
  title={Inverse Stefan Problems},
  author={Gol'dman, Natali{\^a} L'vovna},
  volume={412},
  year={2012},
  publisher={Springer Science \& Business Media}
}

\end{document}